\documentclass{amsart}
\usepackage{amssymb,amsmath}
\usepackage{amsfonts}
\usepackage{pstricks, pst-node}
\usepackage[all]{xy}
\numberwithin{equation}{section}
\begin{document}
\title{SPIN GRAPHS}
\author{K.~M.Bugajska}
\address{Department 0f Mathematics and Statistics,
York University,
Toronto, ON, M3J 1P3}
\begin{abstract} 
We show that on any Riemann surface $\Sigma$ of genus $g>1$ any non-singular even spin bundle $\xi_{\epsilon}$  defines an $\epsilon$-foliation of $\Sigma$.  When a surface is hyperelliptic then all leaves of this foliation are finite and almost all of them consist of $2g+2$ points.  Moreover, each leaf carries an additional structure which allows us to view it as a concrete graph.  We find the properties of these spin-graphs and we describe the classification of  surfaces which is given by these properties.  The classification  is based on a finite number of exceptional graphs which have to be present on any  surface $\Sigma$ of genus $g\geq{2}$.
\end{abstract}

\maketitle

\section{INTRODUCTION}

We will show that for any point ${P}\in{\Sigma}$ any non-singular even spin bundle $\xi_{\epsilon}$  on $\Sigma$   determines the spin-graph 
${\mathcal{S}}^{\epsilon}_{P}$.  For almost all non Weierstrass points of ${\Sigma}$ the spin-graphs are isomorphic to each other, that is, they have the same generic form $\mathcal{S}(g)$ , where g is the genus of a surface $\Sigma$. In addition, there are two spin-graphs through the Weierstrass points and besides we must  have "`exceptional"' spin graphs which involve "`exceptional points"' (at most 4g of them).

In this paper we will consider only hyperelliptic Riemann surfaces. The image of any point $P$ of a surface $\Sigma$ under the hyperelliptic involution is denoted by $\widetilde{P}$. The set of the Weierstrass points (i.e.points with the property $P=\widetilde{P}$) will be denoted by $\mathcal{W}$.

Let $\mathcal{S}^{\epsilon}_{P}$ be a spin graph through a point $P\in{\Sigma}$  and let $\{\mathcal{S}^{\epsilon}_{P}\}$ denote the set of its vertices.  Any spin-graph ${\mathcal{S}}^{\epsilon}_{P}$ consists of vertices, edges, faces and n-cells ($n\leq{(g-1)}$).  For each  vertex $Q\epsilon{\{\mathcal{S}^{\epsilon}_{P}\}}$  the graph ${\mathcal{S}}^{\epsilon}_{P}$  allows us  to read the divisor of the unique meromorphic section of $\xi_{\epsilon}$ with the single, simple pole at $Q$.  Besides, it is shown in ~\cite{KB13}, ~\cite{KMB13}, that $\mathcal{S}^{\epsilon}_{P}$ carries a structure which  allows us to associate to any of its vertex $Q$ a group $\mathcal{G}_{\epsilon}(Q)$  of permutations acting transitively on the set $\widehat{Q}_{\epsilon}:=\Phi_{Q}(\Sigma)\cap{\Theta}_{\epsilon}$. (Here $\Phi_{Q}:{\Sigma}\rightarrow{Jac{\Sigma}}$ denotes the Jacobi mapping with the origin at $Q$ and $\Theta_{\epsilon}$ is the divisor of the theta function $\Theta[\epsilon]$ on Jac$\Sigma$. ~\cite{FK01}, ~\cite{DM84}) 

A generic spin graph $\mathcal{S}(g)$ has $2g+2$ vertices and all edges between the appropriate points are simple and straight. Each of two graps through the Weierstrass points has $g+1$ vertices which are connected to each other by straight, simple edges. The vertices of an exceptional spin graph will be called $\epsilon$-exceptional points.  Any exceptional spin graph on a surface of genus $g$ is characterized by an integer $0\leq{r}\leq{g-1}$.

We will find all possible isomorphic classes of exceptional spin-graphs on a surface of genus $g$ and we will describe properties of such graphs that allow us to classify hyperelliptic surfaces.  More precisely, to each surface $\Sigma$  equipped with a nonsingular even spin structure $\xi_{\epsilon}$ we will associate an $M$-tuple of integers $m_{r,s}$  which have to satisfy some condition  (see $(4.1)$ and $(4.6)$).   Here,  $r=0,1,\ldots{g-1}$;  $s=0,1,\ldots{N(r)}$ with  $N(r)$  equal to the number $\sigma_{r+1}(g+1)$ of unordered partitions of $g+1$ into $r+1$ positive integers and
\begin{equation}
M=\sum_{r=0}^{g-1}{N(r)}
\end{equation}  
The general classification of exceptional spin-graps on a surface of genus $g$ will be given. We will illustrate this by giving  explicite examples of all classes of isomorphic exceptional graphs for $g$ equal to $2,3,4$ and $5$.  There are exactly two classes of isomorphic exceptional spin graphs on a surface of genus $2$ and hence exactly three different types of surfaces of genus $2$. There are four isomorphic classes of possible exceptional graphs on a surfaces of genus $3$ and hence, by the condition $(1.1)$, there are exactly nine different types of surfaces of genus $3$. 

When $g=4$  the number of isomorphic classes of exceptional spin graphs is six. Now, by the condition $(4.1)$, each possible $6$-tuple $(m_{0},m_{1,1},m_{1,2},m_{2,1},m_{2,2},m_{3})$ that satisfies  $16=8m_{0}+6(m_{1,1}+m_{1,2})+4(m_{2,1}+m_{2,2})+2m_{3}$ determines a concrete type of a hyperelliptic Riemann surface equipped with a nonsingular spin structure $\xi_{\epsilon}$.

The number $M=\sum_{r=0}^{g-1}(N(r))$ is the same for all surfaces of genus $g$ and for all possible nonsingular even spin-structures. It is an open question whether, for a given surface $\Sigma$, the $M$-tuples $(m_{r,s})$, whith $r=0,\ldots,{g-1}; s=0,\ldots,{N(r)}$; are the same or they are different for different non-singular even characteristics $[\epsilon]'s$.

More subtle classification of hyperelliptic Riemann surfaces which uses not only exceptional spin graphs but also exceptional spin groups associated to each  vertex $Q$ of an exceptional spin-graph $\mathcal{S}_{P}^{\epsilon}$ is given in the paper [2].

\section{PRELIMINARIES}
Let $\Sigma$ be a compact, hyperelliptic Riemann surface of genus $g\geq{2}$ and let $\xi_{\epsilon}$ be a nonsingular even spin bundle over $\Sigma$. For any point $P\in{\Sigma}$ there exists (see  ~\cite{DV11}, ~\cite{RG67} ) a unique point-like representation of $\xi_{\epsilon}$, namely:

\begin{equation}
\xi_{\epsilon}\cong{{{\xi_{P}}^{-1}}\otimes{\xi_{P_{1}}}\otimes{\ldots}\otimes{\xi_{P_{g}}}};
\qquad         {P}\notin\{{P}_{1},{P}_{2},\ldots{P}_{g}\}
\end{equation}
with the property that the index of specialty of the divisor
\begin{equation}
{\mathcal{A}_{P}}^{\epsilon}:={P_{1}P_{2}\ldots{P_{g}}}
\end{equation}
vanishes. (Here the points ${P_{1}},\ldots,{P_{g}}$ are not necessarily distinct.) Equivalently,  there exists unique (up to a nonzero multiplicative constant) section $\sigma_{P}$ of the line bundle $\xi_{\epsilon}$ with a single, simple pole at $P$ and hence  with the  divisor given by 
 $(\sigma_{P}):=div\sigma_{P}=P^{-1}{\mathcal{A}_{P}}^{\epsilon}$.
 
 Let $\mathcal{D}$ be any integral divisor on $\Sigma$. By   $\{\mathcal{D}\}$ we will denote the set of all distict points of the divisor $\mathcal{D}$. Moreover, for any divisor $\mathcal{U}$ on a hyperelliptic surface $\Sigma$   the divisor that is obtained from $\mathcal{U}$ by replacing all of its points by their conjugate respectively, will be denoted by $\widetilde{\mathcal{U}}$. 
 
  Suppose that we have fixed a point $P\in{\Sigma}$. Let ${P_{i}}\in\{{\mathcal{A}_{P}}^{\epsilon}\}$. Now, for each $i=1,\ldots,g$, the point $P_{i}$ determines a new set of points $P_{i,j}$, $j=1,\ldots,g$, which form the divisor $\mathcal{A}_{P_{i}}^{\epsilon}$ respectively.  We will continue this process for all such points obtained in the previous step.
 \newtheorem{definition}{Definition} 
 \begin{definition} 
 The set of all points of $\Sigma$ containing  a point $P$ and obtained in the way described above starting from  the set $\{\mathcal{A}_{P}^{\epsilon}\}$ will be denoted by $\{\mathcal{S}_{P}^{\epsilon}\}$.
 \end{definition}

 \begin{definition}
 For any point $Q\in{\Sigma}$ the cardinality of the set $\{\mathcal{A}_{Q}^{\epsilon}\}$ will be denoted by  $ deg_{\epsilon}{Q}$ and will be called the $\epsilon$-degree of the point $Q$.
 \end{definition}
 \newtheorem{lemma}{Lemma}
 \begin{lemma}\hfill
 Let $P$ be any point of a hyperelliptic Riemann surface $\Sigma$. Then
 \begin{enumerate}
 \item ${\mathcal{A}_{\widetilde{P}}^{\epsilon}}={\widetilde{\mathcal{A}_{P}^{\epsilon}}}$
 \item For any $P_{i}\in\{\mathcal{A}_{P}^{\epsilon}\}$ we have $P\in{\{\mathcal{A}_{P_{i}}^{\epsilon}\}}$
 \end{enumerate}
 \end{lemma}
 \begin{proof}
 The property $(1)$ follows immediately when we  consider a meromorphic differential given by the product of two sections $\sigma_{P}$ and $\sigma_{\widetilde{P}}$. Analogously, by considering a section $\sigma_{P}\sigma_{P_{i}}$ of the canonical bundle ${\xi}_{\epsilon}^{\otimes{2}}$ we obtain the property $(2)$.
 \end{proof}
 
 \begin{definition}(~\cite{RG67},  ~\cite{RK95}, ~\cite{FK92})
 Let $\mathcal{D}_{k}^{P}$ denote the integral divisor that is obtained from the divisor $\mathcal{A}_{P}^{\epsilon}$ by deleting the point ${P_{k}}\in\{\mathcal{A}_{P}^{\epsilon}\}$. The unique integral divisor that is complementary to $\mathcal{D}_{k}^{P}$ will be denoted by $\widehat{\mathcal{D}_{k}^{P}}$ (i.e. the product  $\mathcal{D}_{k}^{P}{\widehat{\mathcal{D}_{k}^{P}}}$ is a canonical divisor).  
 \end{definition}
 
 \newtheorem{corollary}{Corollary}
 \begin{corollary}\hfill
 \begin{enumerate}
 
 \item When a point $P\in{\Sigma}$ has $\epsilon$-degree  equal to $g$ then we have ${\widehat{\mathcal{D}_{k}^{P}}}={\widetilde{\mathcal{D}_{k}^{P}}}$.
 \item When $deg_{\epsilon}{P}=g$ then the divisors 
 \begin{equation*}
 {P^{-1}{\mathcal{A}_{P}^{\epsilon}}}={{P^{-1}}{P_{1}}\ldots{P_{g}}} \qquad and \qquad {P_{k}^{-1}}{P}{\widetilde{P}_{1}}\ldots{\widehat{\widetilde{P_{k}}}}\ldots{\widetilde{P}_{g}} 
 \end{equation*}
 are equivalent to each other. (Here  the hat means that the point $\widetilde{P}_{k}$ is omitted.)
 \end{enumerate}
 \end{corollary}
 
 \begin{lemma}\hfill
 Suppose that  point $P\in{\Sigma}$ has $\epsilon$-degree equal to $g$.
 \begin{enumerate}
 \item When ${\widetilde{P}}\notin{\{\mathcal{A}_{P}^{\epsilon}\}}$ then  each point ${Q}\in{\{\mathcal{S}_{P}^{\epsilon}\}}$  has  $\epsilon$-degree equal to $g$ and we must have ${Q}\notin{\{\mathcal{A}_{\widetilde{Q}}^{\epsilon}\}}$.
\item When ${\widetilde{P}}\in{\{{\mathcal{A}_{P}^{\epsilon}}\}}$ then there exists ${Q}\in{\{\mathcal{S}_{P}^{\epsilon}\}}$ with $\epsilon$-degree necessarily less than $g$;  ${deg_{\epsilon}{Q}}<{g}$.
 \end{enumerate}
 \end{lemma}
 \begin{proof}
 Simple.
 \end{proof}
 \begin{corollary}
 When the $\epsilon$-degree of a non-Weierstrass point $P\in{\Sigma}$ is equal to $g$ and when $\widetilde{P}\notin{\{\mathcal{A}_{P}^{\epsilon}\}}$ then the cardinality of the set $\{\mathcal{S}_{P}^{\epsilon}\}$ is $2g+2$.  More precisely we have
 \begin{equation*}
 {\{\mathcal{S}_{P}^{\epsilon}\}}={\{P,P_{k},\widetilde{P},\widetilde{P}_{k};k=1\ldots{g}\}}
 \end{equation*}
 \end{corollary}
 \begin{lemma}
 For any Weierstrass point $P\in{\mathcal{W}}\subset{\Sigma}$ we have $deg_{\epsilon}{P}=g$. Moreover,  all points $P_{k}\in{\{\mathcal{A}_{P}^{\epsilon}\}}$ are also the Weierstrass points.
 \end{lemma}
 \begin{proof}
 Since $P={\widetilde{P}}$ we have ${\mathcal{A}_{P}^{\epsilon}}={\mathcal{A}_{\widetilde{P}}^{\epsilon}}={\widetilde{\mathcal{A}_{P}^{\epsilon}}}$ and hence the set $\{\mathcal{A}_{P}^{\epsilon}\}$ is contained in the set $\mathcal{W}$ of the Weierstrass poits of $\Sigma$. Now, since the $\epsilon$-degree of $P$ less than $g$ impies that the divisor of the section $\sigma_{P}$ is an integral divisor (what is impossible for a non-singular even characteristic $[\epsilon]$) we immediately obtain  $deg_{\epsilon}{P}=g={deg_{\epsilon}{P_{k}}}$ for each poit $P_{k}\in{\{\mathcal{A}_{P}^{\epsilon}\}}\subset{\mathcal{W}}$.
 \end{proof}
 
 \section{STANDARD  and WEIERSTRASS SPIN-GRAPHS}
 Let us fix a non-singular spin bundle $\xi_{\epsilon}$ on a hyperelliptic surface $\Sigma$. Let $Q$ be any point of the set ${\{\mathcal{S}_{P}^{\epsilon}\}}\subset{\Sigma}$ introduced by the definition $1$ above.(From now on the index $\epsilon$ may be omitted.) Notice that  we must have ${\{\mathcal{S}_{P}\}}={\{\mathcal{S}_{Q}\}}$. Moreover, lemma 1 implies that when $P'\in{\mathcal{A}_{Q}}$ for some point $P'\in{\{\mathcal{S}_{P}\}}$ then $Q\in{\{\mathcal{A}_{P'}\}}$.  We will say that points $P'$ and $Q$ are $\epsilon$-connected (or mutually connected by $\xi_{\epsilon}$).

 \begin{definition}
 Let $P$ be an arbitrary point of $\Sigma$. The spin graph $\mathcal{S}_{P}$ through $P$ has vertices given by all points of the set $\{\mathcal{S}_{P}\}$ and it has edges which connect only vertices  that are $\epsilon$-connected. Suppose that vertices $Q$ and $R$ of the graph $\mathcal{S}_{P}$ are $\epsilon$-connected.  When  $n_{1}\geq{1}$ is the maximal integer such that the divisor ${Q^{n_{1}}}<{\mathcal{A}_{R}}$ and when $n_{2}\geq{1}$ is the maximal integer such that $R^{n_{2}}<{\mathcal{A}_{Q}}$ and, for example, $n_{1}\leq{n_{2}={n_{1}+k}}$; $k\geq{0}$,  then the straight edge between $Q$ and $R$ has the multiplicity equal to $n_{1}$ and the spin-graph $\mathcal{S}_{P}$ has the additional, oriented arc edge from $R$ to $Q$ labelled by $k={n_{2}-n_{1}}$.
 \end{definition} 
 \begin{pspicture}(-5,-1.4)(5,4)
\psline[linewidth=0.3mm,showpoints=true]%
(-4,0)(-2.6,0)(-2,1)
\psline[linewidth=0.3mm, showpoints=true]%
(-4.7,1)(-4,2)(-2.6,2)
\psline[linewidth=0.3mm]%
(-4.7,1)(-4,0)
\psline[linewidth=0.3mm]%
(-2.6,2)(-2,1)
\rput(-3,-1){\rnode{A}{Pict.1a}}
\rput(-4.3,-0.1){\rnode{a}{$P_{1}$}}
\rput(-2.3,-0.1){\rnode{b}{$\widetilde{P_{2}}$}}
\rput(-5,0.9){\rnode{c}{$P$}}
\rput(-1.7,0.9){\rnode{d}{$\widetilde{P}$}}
\rput(-4.4,2){\rnode{e}{$P_{2}$}}
\rput(-2.2,2){\rnode{f}{$\widetilde{P_{1}}$}}
\psline[linewidth=0.3mm, showpoints=true]%
(1,0)(3,0)(4,1)
\psline[linewidth=0.3mm, showpoints=true]%
(1,2)(3,2)(4,3)
\psline[linewidth=0.3mm]%
(1,0)(1,2)
\psline[linewidth=0.3mm]%
(3,0)(3,2)
\psline[linewidth=0.3mm,showpoints=true, linestyle=dashed]%
(2,1)(2,3)
\psline[linewidth=0.3mm]%
(1,2)(2,3)
\psline[linewidth=0.3mm, linestyle=dashed]%
(1,0)(2,1)
\psline[linewidth=0.3mm]%
(4,1)(4,3)
\psline[linewidth=0.3mm, linestyle=dashed]%
(2,1)(4,1)
\psline[linewidth=0.3mm]%
(2,3)(4,3)
\rput(0.7,0.3){\rnode{g}{$P$}}
\rput(2.7,0.3){\rnode{h}{$P_{1}$}}
\rput(1.7,1.2){\rnode{i}{$P_{2}$}}
\rput(3.7,1.3){\rnode{j}{$\widetilde{P_{3}}$}}
\rput(2.7,2.3){\rnode{k}{$\widetilde{P_{2}}$}}
\rput(3.7,3.3){\rnode{l}{$\widetilde{P_{2}}$}}
\rput(0.7,2.3){\rnode{m}{$P_{3}$}}
\rput(1.7,3.3){\rnode{n}{$\widetilde{P_{1}}$}}

\rput(2.4,-1){\rnode{B}{Pict.1b}}
\end{pspicture}

 We see immediately that when a spin-graph $\mathcal{S}_{P}$ is given then, for each of its vertex $Q$, we can read the divisor of the section $\sigma_{Q}$ of the bundle $\xi_{\epsilon}$.
 \begin{definition}
 A non-Weierstrass point $P$  will be called a standard point  of $\Sigma$ when $deg_{\epsilon}{P}={g}$ and  ${\widetilde{P}}\notin{\{\mathcal{A}_{P}\}}$. 
 \end{definition}
  Lemma $2$ implies that all vertices of the spin-graph $\mathcal{S}_{P}$ through a standard point $P$ are also standard points of $\Sigma$ as well as that all edges of $\mathcal{S}_{P}$ must be simple straight edges.  The spin-graph through a standard point will be called a standard graph. 
  \begin{corollary}
  All standard spin-graphs on any hyperelliptic Riemann surface $\Sigma$ of genus $g$ equipped with any non-singular spin bundle $\xi_{\epsilon}$ are isomorphic to each other. In other words, the isomorphic class of standard spin-graphs depends only on the genus $g$ of a surface.
  \end{corollary}
The standard graphs for genus $g=2$ and $g=3$ are given by the Pict1a and Pict1b respectively.

 Let $P$ be a Weierstrass point. We already know that the $\epsilon$-degree of $P$ must be equal to $g$ and that all points of the divisor $\mathcal{A}_{P}$ are also  Weierstrass points.  Since we have the property that $P\in{\{\mathcal{A}_{Q}\}}$ if and only if $Q\in{\{\mathcal{A}_{P}\}}$ the graph $\mathcal{S}_{P}$ through  a  point $P$ has $g+1$ vertices (which all are  Weierstrass points) and all of them are mutually connected by straight, simple (i.e. with multiplicity equal to $1$) edges. The examples of such graphs for genus $g=2$ and for genus $g=3$  are given by the Pict.2a and by Pict.2b respectively.

\begin{pspicture}(-4.5,-1.5)(4.5,4.5)
\psline[linewidth=0.3mm, showpoints=true]
(-4,0)(-2,0)(-3,1.73)
\psline[linewidth=0.3mm]
(-4,0)(-3,1.73)
\rput(-3,-1.3){\rnode{A}{Pict.2a}}
\rput(-4.3,0){\rnode{a}{$P$}}
\rput(-1.7,0){\rnode{b}{$P_{1}$}}
\rput(-2.7,1.73){\rnode{c}{$P_{2}$}}

\psline[linewidth=0.3mm, showpoints=true]
(1,0)(4,0)(2,3.8)
\psline[linewidth=0.3mm]%
(1,0)(2,3.8)
\psline[linewidth=0.3mm, showpoints=true,linestyle=dashed]
(2,3.8)(2.5,1)
\psline[linewidth=0.3mm, linestyle=dashed]
(1,0)(2.5,1)
\psline[linewidth=0.3mm,linestyle=dashed]
(4,0)(2.5,1)
\rput(2.5,-1.3){\rnode{B}{Pict.2b}}
\rput(0.7,0){\rnode{d}{$P$}}
\rput(4.3,0){\rnode{e}{$P_{1}$}}
\rput(2.2,1.1){\rnode{f}{$P_{2}$}}
\rput(1.7,3.8){\rnode{g}{$P_{3}$}}

\end{pspicture}

 \section{EXCEPTIONAL SPIN-GRAPHS}
 \subsection{Generalities}
 Any non Weierstrass point $Q\in{\Sigma}$ that is not a standard point will be called an exceptional point of $\Sigma$.  It occurs that on any hyperelliptic Riemann surface we must have points whose $\epsilon$-degree is equal to $\rho<g$ (i.e. exceptional points). To show this let us consider the unique (up to a non zero multiplicative constant) meromorphic function $f_{P}$ which connects two sections $\sigma_{P}$ and $\sigma_{\widetilde{P}}$ of the line bundle $\xi_{\epsilon}$; here $P$ is an arbitrary, fixed standard point of $\Sigma$. Fom the relation ${\sigma_{\widetilde{P}}}={f_{P}}\sigma_{P}$ we see that the function $f_{P}$ has degree equal to $g+1$ and its divisor is
 \begin{equation*}
 (f_{P})=({\frac{\sigma_{\widetilde{P}}}{\sigma_{P}}})={\frac{P{\widetilde{P_{1}}}\ldots{\widetilde{P_{g}}}}{\widetilde{P}P_{1}\ldots{P_{g}}}}
 \end{equation*}
 We see immediately that the ramification number of the mapping $f_{P}:{\Sigma}\rightarrow{\widehat{\mathbb{C}}}$ at any standard or at any Weierstrass point of $\Sigma$ is equal to $1$. 
 
 Let $Q$ be any point of $\Sigma$.  Let $\mathbb{A}_{Q}$ denote the integral divisor $Q{\mathcal{A}_{\widetilde{Q}}}$ and let $\{\mathbb{A}_{Q}\}$ denote, as usually, the set of its distinct points.
 \begin{lemma}\hfill
 Let $P$be a standard point and let $f_{P}$ be the meromorphic function on $\Sigma$ introduced above. Then
 \begin{enumerate}
 \item For each point $Q\in{\Sigma}$ the function $f_{P}$ is constant on the set $\{\mathbb{A}_{Q}\}$. 
 \item If $Q$ is any standard point different than $P$ than functions $f_{P}$ and $f_{Q}$ are related by some Moebius transformation.
 \item There are points on $\Sigma$ whose $\epsilon$-degree is equal to $\rho<g$. There are at most $4g$ such points.
 \end{enumerate}
 \end{lemma}
 \begin{proof}\hfill
 \begin{enumerate}
 \item Suppose that $f_{P}(Q)={z_{1}}\in{{\mathbb{C}}^{*}}$ (i.e.point $Q\notin{\{\mathcal{S}_{P}\}}$). Since the index of specialty $i(\mathcal{A}_{P})=0$ we see that the zero divisor of the function $f_{P}-{z_{1}}$ is $(f_{P}-{z_{1}})^{0}={Q\mathcal{A}_{\widetilde{Q}}}={\mathbb{A}_{Q}}$. This means that we have  $f_{P}(R)={z_{1}}=f_{P}(Q)$ for each point $R\in{\{\mathcal{A}_{\widetilde{Q}}\}}$.
 \item It is a simple consequence of the fact that whenever $Q$ is a  standard point  with $f_{P}(Q)={z_{1}}\neq{0}$ then $f_{P}(\widetilde{Q})={z_{2}}\in{\mathbb{C}^{*}}$ and $z_{1}\neq{z_{2}}$. Hence the divisor of the function $(\frac{f_{P}-{z_{1}}}{f_{P}-{z_{2}}})$ is equal to the divisor ${\frac{\mathbb{A}_{Q}}{\mathbb{A}_{\widetilde{Q}}}}=(f_{Q})$.
 \item By the Hurwitz-Riemann theorem the total branch number $B$ of the function $f_{P}$ is $B=4g$. Since neither the Weierstrass points nor standard points can be ramification points of $f_{P}$ we must have some other points which have  non-zero branching numbers. However, a non-zero branching number at a point $Q\in{\Sigma}$ means that this point occurs with the multiplicity $m>1$ in the divisor $\mathcal{A}_{R}$ for some point $R\in{\{\mathcal{A}_{Q}\}}$. In other words, the $\epsilon$-degree of $R$ must be smaller than $g$. So, on any hyperelliptic surface $\Sigma$ there are exceptional points and there are at most $4g$ of them.
 \end{enumerate}
 \end{proof}
 The spin-graph through an exceptional point will be called an exceptional graph. The number of such graphs is restricted by the value of the total branching number $B=4g$.  Contrary to the fact that all standard graphs belong to exactly  one isomorphic class of graphs (depending only on the genus $g$ of $\Sigma$)  and that the same is true for the Weierstrass spin graphs, the exceptional spin graphs may belong to distinct isomorphic classes.
 
 Suppose that on a surface $\Sigma$ of genus $g$ we have  some number $M$ of possible isomorphic classes of exceptional spin graphs. Let $B_{i}$,  $i={1,\ldots,N}$ denote the total branch number carried by the $i$-th class of such graphs.  This number  is uniquely determined by the multiplicities of edges occuring in the graph. Let $m_{i}$  denote the number of exceptional graphs on $\Sigma$ that belong to the $i$-th class. We observe that we may classify hyperelliptic Riemann surfaces of genus $g$ equipped with a non-singular even spin strusture $\xi_{\epsilon}$ by an ordered arrays of $M$ nonnegative integers $(m_{1},m_{2},\ldots,{m_{M}})$ which satisfy
 \begin{equation}
 4g={\sum}_{i=1}^{M}{m_{i}B_{i}}
 \end{equation}
From our general considerations above we may notice that for any point $Q\in{\Sigma}$ and for any vertex $R$ of the graph $\mathcal{S}_{Q}$ we have 
\begin{equation}
 \text {either} \quad {\mathbb{A}_{R}}={\mathbb{A}_{Q}} \quad  \text{or} \quad {\mathbb{A}_{R}}={\mathbb{A}_{\widetilde{Q}}}
\end{equation} 

\begin{lemma}
Suppose that the integral divisor $\mathcal{A}_{Q}$ corresponding to an exceptional point $Q\in{\Sigma}$ is given by
\begin{equation*}
  \mathcal{A}_{Q}={{\widetilde{Q}^{k_{0}-1}}Q_{1}^{k_{1}}\ldots{Q_{r}^{k_{r}}}}; \qquad k_{0}\geq{1} \qquad (k_{0}-1)+k_{1}+\ldots{+k_{r}}=g
 \end{equation*} 
     with $r<g$.  The total branch number carried by the spin graph $\mathcal{S}_{Q}$ is equal to $B(\mathcal{S}_{Q})={2(g-r)}$.
\end{lemma}
\begin{proof}
The form of the divisor $\mathcal{A}_{Q}$ implies that the divisor $\mathbb{A}_{Q}:={Q\mathcal{A}_{\widetilde{Q}}}$ may be written as follows
\begin{equation}
{\mathbb{A}_{Q}}=Q^{k_{0}}{\widetilde{Q}_{1}^{k_{1}}}{\widetilde{Q}_{2}^{k_{2}}}\ldots{\widetilde{Q}_{r}^{k_{r}}} \qquad with \qquad k_{0}+k_{1}+k_{2}+\ldots{+k_{r}}=g+1
\end{equation}
 The relation (4.2) implies that for each vertex $R\in{\{\mathcal{S}_{Q}\}}$ the divisor $\mathbb{A}_{R}$  (which has  the same form as  (4.3))  may have  distinct values of $k_{i}; i=0,1,\ldots,r$,  but the number $r$ of the   remaining points ( different than $R$)  in the set  $\{\mathbb{A}_{R}\}$ is  exactly the same   for  every vertex $R$. This means that we may characterize any exceptional spin graph $\mathcal{S}_{P}$ by an integer $r<g$. Besides, we see that  to evaluate the total branch number $B(\mathcal{S}_{Q})$ we may use an arbitrary vertex of the graph $\mathcal{S}_{Q}$ . More precisely we have 
\begin{equation}
B(\mathcal{S}_{Q})=2[(k_{0}-1)+(k_{1}-1)+\ldots{+(k_{r}-1)}]=2(g-r)
\end{equation}
as expected.
\end{proof}

\begin{definition}
The integer $r<g$ introduced for any exceptional spin graph $\mathcal{S}_{P}$ in the  lemma above will be called the order of this exceptional graph.
\end{definition} 
\begin{definition}
 For each  vertex $R$ of an exceptional spin graph $\mathcal{S}_{Q}$ we introduce the $r+1$-tuple of positive integers $(k_{0},k_{1},\ldots,{k_{r}})$ determined by the form (4.3) of the divisor $\mathbb{A}_{R}$, $r<g$. This $r+1$-tuple will be denoted by $\widehat{k}(R)$. A vertex $R$ of a spin graph $\mathcal{S}_{Q}$  will be called a head of the graph if
\begin{equation}
k_{0}(R)=min\{k_{0}(Q');Q'\in{\{\mathcal{S}_{Q}\}}\}
\end{equation}
\end{definition}

\begin{lemma}
Let $\mathcal{S}_{P}$ and $\mathcal{S}_{Q}$ be two different exceptional graphs with heads $P$ and $Q$ respectively. These graphs are isomorphic if and only if, after eventually change of the enumerations of their vertices, $\widehat{k}(P)=(k_{0},k_{1},\ldots,k_{r})$ coincides with  $\widehat{k}(Q)=(k'_{0},k'_{1},\ldots,{k'_{s}})$. 
\end{lemma}
\begin{proof}
When the graphs are isomorphic then obviously we have  $\widehat{k}(P)={\widetilde{k}(Q)}$.  Conversely, the equality $\widehat{k}_{P}={\widehat{k}_{Q}}$ means that $r=s$ and that $k_{l}(P)=k_{l}(Q)$ for $l=0,1,\ldots,r$. Equivalently, the divisors of the sections $\sigma_{P}$ and $\sigma_{Q}$ are: $(\sigma_{P})=P^{-1}{\widetilde{P}^{k_{0}-1}}P_{1}^{k_{1}}\ldots{P_{r}^{k_{r}}}$ and $(\sigma_{Q})=Q^{-1}{\widetilde{Q}^{k_{0}-1}}Q_{1}^{k_{1}}\ldots{Q_{r}^{k_{r}}}$. Now the isomorpism between the graphs $\mathcal{S}_{P}$ and $\mathcal{S}_{Q}$ follows immediately.
\end{proof}
For any vertex $Q$ of any spin graph $\mathcal{S}_{P}$ the integer $r\geq{0}$ denotes the number of points of the set $\{\mathcal{A}_{Q}\}$ that are different than the conjugate point $\widetilde{Q}$.  Although the $\epsilon$-degree ${deg_{\epsilon}Q}\geq{r}$  may be different for different vertices of a given exceptional graph $\mathcal{S}_{P}$,  the  integer $r$ is for all vertices exactly the same and, by the definition $6$, it is called the order of a graph. However, exceptional graphs with the same order $r$,  $0\leq{r}<g$, may belong to different classes of isomorphic spin graphs that are posiible on a given surface of genus $g$. 

This means that the integer $r$ does not define uniquely a class of isomorphic spin graphs.
  More precisely, the  lemma 6 implies that the number of possible different classes of isomorphic spin graphs with a given $0\leq{r}<g$ is  equal to the number $N(r)$ of representations of $g+1$ as a sum ${g+1}={k_{0}+k_{1}+\ldots{k_{r}}}$ of non-decreasing integers $1\leq{k_{0}\leq{k_{1}}\leq{\ldots}k_{r}\leq{g-r-1}}$  i.e. $N(r)={\sigma}_{r+1}(g+1)$.
 
 Now, $N(0)=1$ and it corresponds to the unique class of isomorphic graphs with $\widehat{k}=(k_{0})$ where $k_{0}=g+1$. The value of  $N(g-1)$ is also one and  it corresponds to the unique class with $\widehat{k}=(k_{0},k_{1},\ldots,{k_{g-1}})=(1,1,\ldots,{1},2)$.
 
 Let us fix the genus $g$ and let $\mathcal{S}^{r}_{s}={\mathcal{S}^{r}_{s}(g)}$ denote a class of appropriate exceptional graphs with a given $r$ and with $s\in{\{1,2,\ldots,{N(r)}\}}$. Let $m_{r,s}$ denote the number of such graphs that occur on a given hyperbolic surface $\Sigma$. According to $(4.1)$ we must have
 \begin{equation}
 4g=m_{0}B_{0}+B_{1}(m_{1,1}+\ldots{+m_{1,N(1)}})+\ldots{B_{r}(m_{r,1}+\ldots{+m_{r,N(r)}})}+\ldots{+B_{g-1}m_{g-1}}
\end{equation}
Since $M$, given by the formula (1,1), determines the number of possible different classes of isomorphic  exceptional spin graphs on a surface of the genus $g$ we see that we may classify hyperelliptic Riemann surfaces by associating to each of them an $M$-tuple of non-negative integers $(m_{r,s}; r=0,1,\ldots,{g-1}; s=1,\ldots,N(r))$.
\newtheorem{example}{Example}
\begin{example}
Suppose that $g=2$. Since we must have $r\in{\{0,1\}}$ the possible classes $\mathcal{S}^{r}_{s}$ of exceptional graphs are $\mathcal{S}^{0}_{1}$ and $\mathcal{S}^{1}_{1}$. Hence, to any surface of genus $2$ we may associate a pair $(m_{0},m_{1})$ of integers which has to satisfy the condition $8=4m_{0}+2m_{1}$. We see that we may have exactly three different types of such surfaces corresponding to $(m_{0},m_{1})$ equal either to  $(2,0)$ or to $(1,2)$ or to $(0,4)$ respectively.
\end{example}
\begin{example}
Let $g=3$. Now $r\in{\{0,1,2\}}$. A possible exceptional graph $\mathcal{S}_{P}$ with a head $P$ may belong to the following classes: either to $\mathcal{S}^{0}_{1}$ (when $\mathcal{A}_{P}={\widetilde{P}^{3}}$), or to $\mathcal{S}^{1}_{1}$ (when $\mathcal{A}_{P}=P_{1}^{3}$),or to $\mathcal{S}^{1}_{2}$ (when $\mathcal{A}_{P}={\widetilde{P}P_{1}^{2}}$) or to the class $\mathcal{S}^{2}_{1}$ (when $\mathcal{A}_{P}=P_{1}P_{2}^{2}$). Since the condition $(4.6)$ requires that $12=6m_{0}+4(m_{1,1}+m_{1,2})+2m_{2}$ any surface of genus $g=3$ can be characterized by a quadruple $(m_{0},m_{1,1},m_{1,2},m_{2})_{\epsilon}$ of non negative integers which belongs to the set:
\begin{equation*}
\{(2,0,0,0),(1,1,0,1,),(1,0,1,1),(0,1,2,0),(0,2,1,0),
\end{equation*}
\begin{equation*}
(0.1.1.2),(0,1,0,4),(0,0,1,4),(0,0,0,6)\}
\end{equation*}
So, we have nine different types of surfaces of genus $g=3$ equipped with an even, non-singular spin structure $\xi_{\epsilon}$. 
\end{example}
\begin{example}
Suppose that the genus of a hyperelliptic surface is $g=4$. Now we may have $r\in{\{0,1,2,3\}}$. A  posiible exceptional spin graph $\mathcal{S}_{P}$ with a head $P$ belongs to one of the following classes of isomorphic graphs:  to the class $\mathcal{S}_{1}^{0}$ (when $\mathcal{A}_{P}={\widetilde{P}^{4}}$), or to $\mathcal{S}_{1}^{1}$ (when $\mathcal{A}_{P}=P_{1}^{4}$), or to $\mathcal{S}^{1}_{2}$ (when $\mathcal{A}_{P}={\widetilde{P}P_{1}^{3}}$) or to the class $\mathcal{S}^{2}_{1}$ (when $\mathcal{A}_{P}=P_{1}P_{2}^{3}$), or to $\mathcal{S}^{2}_{2}$ (when $\mathcal{A}_{P}=P_{1}^{2}P_{2}^{2}$), or to the class $\mathcal{S}^{3}_{1}$ (when $\mathcal{A}_{P}=P_{1}P_{2}P_{3}^{2}$). Thus, to any surface of genus $g=4$ we will associate a six-tuple $(m_{0},m_{1,1},m_{1,2},m_{2,1},m_{2,2},m_{3})_{\epsilon}$ of non negative integers which satisfy
\begin{equation*}
16=8m_{0}+6(m_{1,1}+m_{1,2})+4(m_{2,1}+m_{2,2})+2m_{3}
\end{equation*}
Each non-negative integer $m_{r,s}$ indicates the number of exceptional spin graphs that occurs on $\Sigma$ which belong to the isomorphic class $\mathcal{S}^{r}_{s}$; where $r=0,1,2,3$ and $s=1,..,N(r)$.
\end{example}

\subsection{Exceptional graphs for arbitrary genus $g$}
Let $\mathcal{S}_{P}$ be an exceptional graph on a surface $\Sigma$ of genus $g$ whose head is $P$. Suppose that the divisor of the section $\sigma_{P}$ of $\xi_{\epsilon}$ is
\begin{equation}
(\sigma_{P})=P^{-1}{\widetilde{P}^{i}}P_{1}^{k_{1}}\ldots{P_{r}^{k_{r}}}; \quad \text{with}\quad i=k_{0}-1,\quad 0\leq{r}\leq{g-1} 
\end{equation}
\begin{equation*}
\text{and with} \quad k_{0}\leq{k_{1}}\leq{\ldots}{\leq}{k_{r}}
\end{equation*}
We will introduce the following notation:
\begin{equation*}
k_{0}=1+i, \quad k_{n}=k_{n-1}+p_{n}; \quad n=1,2,\ldots,r
\end{equation*}
Now, the class ${\mathcal{S}^{r}_{s}}={\mathcal{S}^{r}_{s}(g)}$ of isomorphic exceptional graphs will be denoted by $\mathcal{S}_{i,p_{1},\ldots,p_{r}}$.  Since from $(4.3)$ we have  
\begin{equation*}
 {g+1}=(r+1)(1+i)+rp_{1}+(r-1)p_{2}+{\ldots}+p_{r}
 \end{equation*} 
 we see that the maximal possible value of $i=k_{0}-1$ is 
\begin{equation}
i_{max}=\left\lfloor {\frac{g-r}{r+1}}\right\rfloor
\end{equation}
\subsubsection{r=1}
Let $P$ be a head of an exceptional spin graph with $r=1$ i.e. the divisor $\mathbb{A}_{P}=P^{k_{0}}{\widetilde{P}_{1}^{k_{1}}}$  (equivalently $\mathcal{A}_{P}={\widetilde{P}^{i}}P_{1}^{k_{1}}$) and $g=i+k_{1}$. Since $P$ is a head we must have $k_{0}\leq{k_{1}}$ and hence 
\begin{equation}
 0\leq{i}\leq{\left\lfloor\frac{g-1}{2}\right\rfloor} 
\end{equation} 
The class $\mathcal{S}_{i,p_{1}}(g)={\mathcal{S}_{i,p_{1}}}$ of isomorphic graphs corresponds to ${\widehat{k}}=(k_{0},k_{1})=(1+i,1+i+p_{1})$.  When $i=0$  then $k_{1}=g$ and the graph $\mathcal{S}_{0,g-1}$ is given on Pict3a.
 
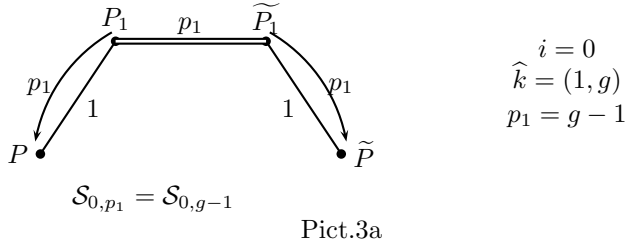
\begin{pspicture}(-4.5,-1.5)(4.5,3)
\psline[ showpoints=true]%
(-4,0)(-3,1.5)
\psline[ showpoints=true]%
(-1,1.5)(0,0)
\psline[doubleline=true]
(-3,1.5)(-1,1.5)
\rput(3,1){\rnode{a}{$\widehat{k}=(1,g)$}}
\rput(3,1.4){\rnode{l}{$i=0$}}
\rput(3,0.5){\rnode{f}{$p_{1}=g-1$}}
\rput(-4.3,0){\rnode{b}{$P$}}
\rput(0.3,0){\rnode{c}{$\widetilde{P}$}}
\rput(-3,1.8){\rnode{d}{$P_{1}$}}
\rput(-1,1.8){\rnode{e}{$\widetilde{P_{1}}$}}
\psarc{->}%
(-2,-0.2){2.1}{120}{170}
\psarc{<-}%
(-2,-0.2){2.1}{10}{60}
\rput(-2,1.7){\rnode{g}{$p_{1}$}}
\rput(-4,0.9){\rnode{h}{$p_{1}$}}
\rput(-3.3,0.6){\rnode{i}{$1$}}
\rput(-0.7,0.6){\rnode{j}{$1$}}
\rput(0,0.9){\rnode{k}{$p_{1}$}}

\rput(-2.5,-0.6){\rnode{A}{${\mathcal{S}_{0,p_{1}}}={\mathcal{S}_{0,g-1}}$}}
\rput(0,-1){\rnode{B}{Pict.3a}}
\end{pspicture}

When $i\geq{1}$ then all conjugate vertices must be connected and the general form of the grap $\mathcal{S}_{i,p_{1}}$ is given by Pict3b.

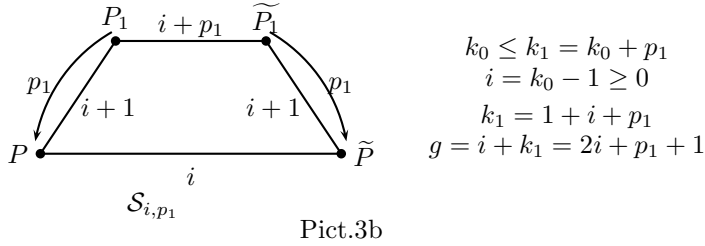
\begin{pspicture}(-4.5,-1.5)(4.5,3)
\psline[ showpoints=true]%
(-4,0)(-3,1.5)
\psline[ showpoints=true]%
(-1,1.5)(0,0)
\psline
(-3,1.5)(-1,1.5)
\psline
(-4,0)(0,0)
\psarc{->}%
(-2,-0.2){2.1}{120}{170}
\psarc{<-}%
(-2,-0.2){2.1}{10}{60}
\rput(-4.3,0){\rnode{e}{$P$}}
\rput(0.3,0){\rnode{f}{$\widetilde{P}$}}
\rput(-3,1.8){\rnode{g}{$P_{1}$}}
\rput(-1,1.8){\rnode{h}{$\widetilde{P_{1}}$}}
\rput(-2,1.7){\rnode{i}{$i+p_{1}$}}
\rput(-4,0.9){\rnode{j}{$p_{1}$}}
\rput(-3.1,0.6){\rnode{k}{$i+1$}}
\rput(-0.9,0.6){\rnode{l}{$i+1$}}
\rput(0,0.9){\rnode{m}{$p_{1}$}}
\rput(-2,-0.3){\rnode{n}{$i$}}

\rput(3,1.4){\rnode{a}{$k_{0}\leq{k_{1}}=k_{0}+p_{1}$}}
\rput(3,1){\rnode{b}{$i={k_{0}-1}\geq0$}}
\rput(3,0.5){\rnode{c}{$k_{1}=1+i+p_{1}$}}
\rput(3,0.1){\rnode{d}{$g={i+k_{1}}=2i+p_{1}+1$}}
\rput(-2.5,-0.7){\rnode{o}{$\mathcal{S}_{i,p_{1}}$}}
\rput(0,-1){\rnode{p}{Pict.3b}}
\end{pspicture}

In particular, when the genus is odd then for $i_{max}=\frac{g-1}{2}$ we have $p_{1}=0$  i.e. $k_{0}=k_{1}$. For the remaining values of $i\leq\frac{g-1}{2}$ we have $p_{1}=2i_{max}-2i\geq{2}$ and hence it is always even. When the genus $g$ is an even integer then $i_{max}<{\frac{g-1}{2}}$ and the corresponding to it $p_{1}$ must be equal to $1$. For the remaining possible values of $i$ we have $p_{1}=2i_{max}-2i+1$ which is always odd.

When $g=3$ then there are two possible classes of isomorphic exceptional graphs corresponding to the values of $(i,p_{1})$ equal to $(0,2)$ and to $(1,0)$.
When $g=4$ then the possibile  graphs $\mathcal{S}_{i,p_{1}}$ correspond to $(i,p_{1})\in\{(0,3),(1,1)\}$. When $g=5$ then $(i,p_{1})\in{\{(04),(1,2),(2,0)\}}$. For all of these values of the genus $g$  we may use the pictures Pict3a and Pict3b to draw appropriate spin graphs. 

 For a general $g\geq{2}$ the total number of possible isomorphic classes $\mathcal{S}_{i,p_{1}}$   is equal to 
\begin{equation*}
N(1)=\sigma_{2}(g+1) 
\end{equation*}
 More precisely, when $g$ is even then 
\begin{equation*}
\widehat{k}=(k_{0},k_{1})\in{\{(1,g),(2,g-1),\ldots,(\frac{g}{2},{\frac{g}{2}}+1)\}}
\end{equation*}
and when $g$ is odd we have
\begin{equation*}
\widehat{k}=(k_{0},k_{1})\in{\{(1,g),(2,g-1),\ldots,(1+{\left\lfloor {\frac{g}{2}}\right\rfloor},1+{\left\lfloor {\frac{g}{2}}\right\rfloor})\}}
\end{equation*}

\subsubsection{r=2}
Suppose that  $\mathcal{S}_{P}$ is  an exceptional spin graphs  with  a head $P$.  We have $\widehat{k}(P)=(k_{0},k_{1},k_{2})=(1+i,1+i+p_{1},1+i+p_{1}+p_{2})$ with $i,p_{1},p_{2}\geq{0}$ and with $i+k_{1}+k_{2}=g$.  The general form of a spin graph with $r=2$ is given by Pict4. Since
\begin{equation}
3i+2p_{1}+p_{2}+2=g
\end{equation}
 such exceptional graph is possible only when the genus of a surface is $g\geq{3}$. The maximal possible value of $i$ is 
\begin{equation*}
i_{max}=\left\lfloor {\frac{g-2}{3}}\right\rfloor
\end{equation*}
When  $r=2$ then, depending from the genus $g$, we may have one or two possible isomorphic classes of graphs with the value $i=i_{max}$. More precisely:
\begin{itemize}
\item When $g\equiv0mod3$ then $i_{max}=\frac{g}{3}-1$ and the property (4.10) implies that $p_{1}=0$, $p_{2}=1$. Hence there is only one class of exceptional graphs with such $i_{m}\equiv{i_{max}}$, namely $\mathcal{S}_{i_{m},0,1}$.
\item When $g\equiv1mod3$ then $i_{max}={\left\lfloor {\frac{g-2}{3}}\right\rfloor}={\left\lfloor {\frac{g}{3}}\right\rfloor}-1$ and  we have  $2p_{1}+p_{2}=2$. There are two possible classes of isomorphic spin graphs with the maximal value of $i=i_{m}$: $\mathcal{S}_{i_{m},0,2}$ and $\mathcal{S}_{i_{m},1,0}$.
\item When $g\equiv2mod3$ then $i_{m}={\left\lfloor {\frac{g-2}{3}}\right\rfloor}={\left\lfloor {\frac{g}{3}}\right\rfloor}$. Now we must have $p_{1}=p_{2}=0$ and hence the unique class $\mathcal{S}_{i_{m},0,0}$ of graphs with the maximal value $i=i_{m}$.
\end{itemize}

From Pict.4 we see that when $i=0$ then a head $P$ of the exceptional graph $\mathcal{S}_{P}\in{\mathcal{S}_{0,p_{1},p_{2}}}$ is not connected to its conjugate $\widetilde{P}$.  We may have a sytuation where either  only one pair of   conjugate vertices ( i.e. $P_{2}$ and $\widetilde{P}_{2}$) is connected or two pairs $\{P_{i},{\widetilde{P_{i}}}\}$ for $i=1,2$ are connected.  In the first case $\mathcal{S}_{P}$ belongs to the unique class $\mathcal{S}_{0,0,g-2}$ and in the latter case we have $\sigma^{2}_{2}(g)$ possible classes $\mathcal{S}_{0,p_{1},p_{2}}$, (with $p_{1}\geq{1}$ and $k_{1}+k_{2}=g$; $ 2\leq{k_{1}}\leq{k_{2}}$) of isomorphic graphs.(Here $\sigma^{l}_{k}(m)$  denotes the number of representations of $m$ as a sum of $k$ non-decreasing integers that are grater than or equal to $l$.) Summarizing, the number of possible graphs with $i=0$ is equal to $\sigma_{2}(g)=1+\sigma^{2}_{2}(g)$.

When $i>0$ then all pairs of conjugate vertices are connected. Since we have $k_{0}={1+i}\geq{2}$ and $k_{0}\leq{k_{1}}\leq{k_{2}}$; $i+k_{1}+k_{2}=g$, it is easy to see that the number of all possible classes of exceptional graphs is now given by $\sigma^{2}_{3}(g+1)$.

The graph $\mathcal{S}_{P}$ is totally symmetric with respect to all of its vertices (see Pict.5) when $k_{0}=k_{1}=k_{2}=1+i$. This is possible only on a surface whose genus is $g\equiv{2mod3}$ and $g={3i+2}\geq{5}$. In particular, for $g=5$ this occurs when $i=i_{max}=1$, $\widehat{k}=(k_{0},k_{1},k_{2})=(2,2,2)$.

We notice that the cardinality of all possible classes of isomorphic exceptional spin graphs with $r=2$ is
\begin{equation}
N(2)=\sigma_{2}(g)+\sigma^{2}_{3}(g+1)={\sigma_{3}}(g+1)
\end{equation}
as expected.

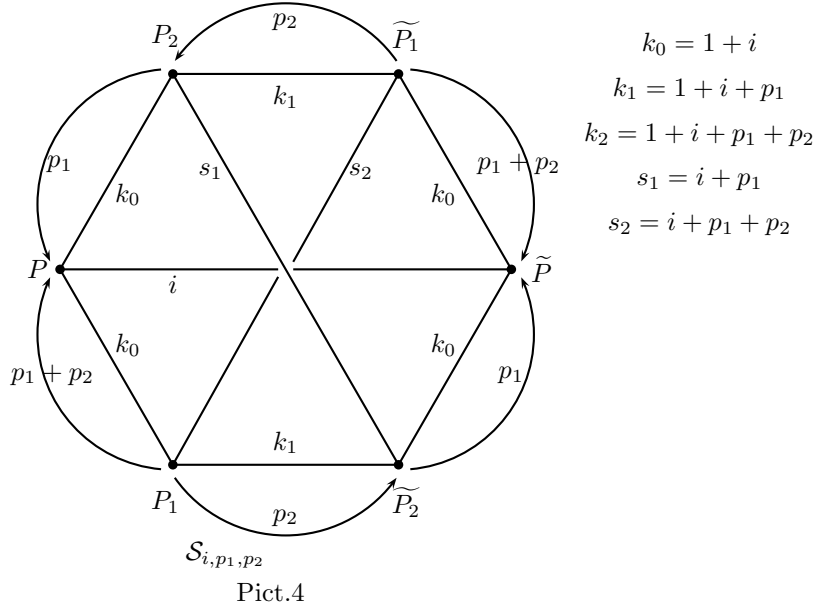
\begin{pspicture}(-4,-4.5)(5.5,4)
\pspolygon[showpoints=true]%
(-1.5,-2.6)(1.5,-2.6)(3,0)(1.5,2.6)(-1.5,2.6)(-3,0)
\rput(5.5,3){\rnode{A}{$k_{0}=1+i$}}
\rput(5.5,2.4){\rnode{B}{$k_{1}=1+i+p_{1}$}}
\rput(5.5,1.8){\rnode{C}{$k_{2}=1+i+p_{1}+p_{2}$}}
\rput(5.5,1.2){\rnode{D}{$s_{1}=i+p_{1}$}}
\rput(5.5,0.6){\rnode{E}{$s_{2}=i+p_{1}+p_{2}$}}
\rput(-0.8,-3.8){\rnode{F}{$\mathcal{S}_{i,p_{1},p_{2}}$}}
\rput(-0.2,-4.3){\rnode{G}{Pict.4}}

\rput(-3.3,0){\rnode{a}{$P$}}
\rput(3.4,0){\rnode{b}{$\widetilde{P}$}}
\rput(-1.6,3.1){\rnode{c}{$P_{2}$}}
\rput(1.6,3.1){\rnode{d}{$\widetilde{P_{1}}$}}
\rput(-1.6,-3.1){\rnode{e}{$P_{1}$}}
\rput(1.6,-3.1){\rnode{f}{$\widetilde{P_{2}}$}}

\psarc {->}%
(-1.5,0.866){1.8}{95}{205}
\psarc{<-}%
(1.5,0.866){1.8}{-25}{85}
\psarc{->}%
(0,1.74){1.8}{35}{145}
\psarc{<-}%
(-1.5,-0.866){1.8}{155}{265}
\psarc {->}%
(1.5,-0.866){1.8}{-85}{25}
\psarc{->}%
(0,-1.74){1.8}{-145}{-35}

\psline(-1.5,2.6)(1.5,-2.6)
\psline(-1.5,-2.6)(-0.1,-0.1)
\psline(0.1,0.1)(1.5,2.6)
\psline(-3,0)(-0.1,0)
\psline(0.1,0)(3,0)

\rput(0,3.3){\rnode{g}{$p_{2}$}}
\rput(0,2.3){\rnode{h}{$k_{1}$}}
\rput(0,-2.3){\rnode{i}{$k_{1}$}}
\rput(0,-3.3){\rnode{j}{$p_{2}$}}
\rput(-1,1.3){\rnode{k}{$s_{1}$}}
\rput(1,1.3){\rnode{l}{$s_{2}$}}
\rput(-1.5,-0.2){\rnode{m}{$i$}}
\rput(-2.1,1){\rnode{n}{$k_{0}$}}
\rput(-2.1,-1){\rnode{o}{$k_{0}$}}
\rput(2.1,1){\rnode{p}{$k_{0}$}}
\rput(2.1,-1){\rnode{q}{$k_{0}$}}
\rput(-3,1.4){\rnode{r}{$p_{1}$}}
\rput(-3.1,-1.4){\rnode{s}{$p_{1}+p_{2}$}}
\rput(3.1,1.4){\rnode{t}{$p_{1}+p_{2}$}}
\rput(3,-1.4){\rnode{u}{$p_{1}$}}

\end{pspicture}

\subsubsection{$r=3$}
Let $P$ be a head of an exceptional spin graph $\mathcal{S}_{P}$ with $r=3$. This means that the section $\sigma_{P}$ of the holomorphic bundle $\xi_{\epsilon}$ has the divisor
\begin{equation*}
 (\sigma_{P})=P^{-1}\mathcal{A}_{P}=P^{-1}{\widetilde{P}}^{i}P_{1}^{k_{1}}P_{2}^{k_{2}}P_{3}^{k_{3}}\quad \text{with} \quad i+k_{1}+k_{2}+k_{3}=g
 \end{equation*}
 Each isomorphic class of graphs with $r=3$ is uniquely determined by a quadruple $(i,p_{1},p_{2},p_{3})$ where $p_{l}=k_{l}-k_{l-1}$ for $l=1,2,3$. Since
 \begin{equation}
 g=4i+3p_{1}+2p_{2}+p_{3}+3
 \end{equation} 
 the genus of a surface that carries such graph must be $g\geq{4}$ and the value of $i=k_{0}-1$ may vary from $0$ to $i_{max}={i_{m}}=\left\lfloor {\frac{g-3}{4}}\right\rfloor$. There are two or one (depending on the genus ) isomorphic classes of graphs with $i=i_{m}$.
 \begin{itemize}
 \item When $g\equiv{3mod4}$ then $i_{m}=\left\lfloor {\frac{g}{4}}\right\rfloor$ and we have only one class $\mathcal{S}_{i_{m},0,0,0}$.
 \item When $g\equiv{2mod4}$ then $i_{m}={\left\lfloor {\frac{g}{4}}\right\rfloor}-1$ and possible classes with $i=i_{m}$ are $\mathcal{S}_{i_{m},1,0,0}$ and $\mathcal{S}_{i_{m},0,0,3}$.
 \item When $g\equiv{1mod4}$ then $i_{m}={\left\lfloor {\frac{g}{4}}\right\rfloor}-1$ again but now we have two possible classes of graphs: $\mathcal{S}_{i_{m},0,1,0}$ and $\mathcal{S}_{i_{m},0,0,2}$.
 \item When $g\equiv{0mod4}$ then  $i_{m}={\frac{g}{4}}-1$ and there is unique possible class of exceptional graphs with $i=i_{m}$, namely  $\mathcal{S}_{i_{m},0,0,1}$
 \end{itemize}
 
 
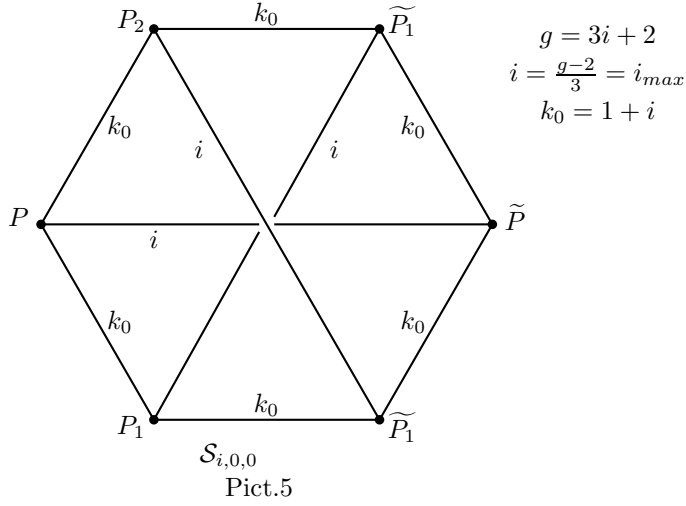
\begin{pspicture}(-4,-4)(5,4)
 \pspolygon [showpoints=true]%
 (-1.5,-2.6)(1.5,-2.6)(3,0)(1.5,2.6)(-1.5,2.6)(-3,0)
 \psline(-1.5,2.6)(1.5,-2.6)
 \psline(-1.5,-2.6)(-0.1,-0.1)
 \psline(0.1,0.1)(1.5,2.6)
 \psline(-3,0)(-0.1,0)
 \psline(0.1,0)(3,0)
 
 \rput(-0.5,-3.1){\rnode{A}{$\mathcal{S}_{i,0,0}$}}
 \rput(-0.1,-3.5){\rnode{B}{Pict.5}}
 
 \rput(-1.8,2.7){\rnode{a}{$P_{2}$}}
 \rput(1.8,2.7){\rnode{b}{$\widetilde{P_{1}}$}}
 \rput(-3.3,0.1){\rnode{c}{$P$}}
 \rput(3.3,0.1){\rnode{d}{$\widetilde{P}$}}
 \rput(-1.8,-2.7){\rnode{e}{$P_{1}$}}
 \rput(1.8,-2.7){\rnode{f}{$\widetilde{P_{1}}$}}
 
 \rput(4.4,2.5){\rnode{C}{$g={3i+2}$}}
 \rput(4.4,2){\rnode{D}{$i={\frac{g-2}{3}}=i_{max}$}}
 \rput(4.4,1.5){\rnode{E}{$k_{0}={1+i}$}}
 
 \rput(0,2.8){\rnode{g}{$k_{0}$}}
 \rput(0,-2.4){\rnode{h}{$k_{0}$}}
 \rput(-1.95,1.3){\rnode{i}{$k_{0}$}}
 \rput(1.95,1.3){\rnode{j}{$k_{0}$}}
 \rput(-1.95,-1.3){\rnode{k}{$k_{0}$}}
 \rput(1.95,-1.3){\rnode{l}{$k_{0}$}}
 
 \rput(-1.5,-0.2){\rnode{m}{$i$}}
 \rput(-0.9,1){\rnode{n}{$i$}}
 \rput(0.9,1){\rnode{o}{$i$}}
 
 \end{pspicture}

We observe that only on a surface of genus $g\equiv{3mod4}$, $g\geq{7}$ we may have an exceptional spin graph whose all pairs of conjugate vertices are connected and the graph is symmetric with respect to all of its vertices (i.e. $k_{0}=k_{1}=k_{2}=k_{3}\geq{2}$)

When $i=0$, i.e. when a head $P$ is not connected with its conjugate, then the number of all possible classes of equivalent graphs is equal to $\sigma_{3}(g)$. More precisely:
\begin{itemize}
 \item When only one pair $P_{3}$ and $\widetilde{P}_{3}$ of conjugate vertises is connected then we have $\widehat{k}(P)=(1,1,1,g-2)=(1,1,1,1+p_{3})$ and $\mathcal{S}_{P}\cong{\mathcal{S}_{0,0,0,g-3}}$.
\item When two pairs, $P_{2}$, $\widetilde{P}_{2}$ and $P_{3}$, $\widetilde{P}_{3}$ are connected then there are $\sigma_{2}(g-1)-1$ possible isomorphic classes $\mathcal{S}_{0,0,p_{2},p_{3}}$, $p_{2}>0$, of exceptional spin graphs.
\item When only a head of an exceptional graph is not connected with its conjugate (i.e. when $p_{1}>0$) then $\widehat{k}(P)=(1,1+p_{1},1+p_{1}+p_{2},1+p_{1}+p_{2}+p_{3})=(k_{0},k_{1},k_{2},k_{3})$ with $2\leq{k_{1}}\leq{k_{2}}\leq{k_{3}}$. The number of possible isomorphic classes of such graphs is $\sigma_{3}^{2}(g)=\sigma_{3}(g)-\sigma_{2}(g-1)$.
\end{itemize}

When $i>0$ then all pairs of conjugate vertices are connected.  Now, for each $1\leq{i}\leq{\left\lfloor {\frac{g-3}{4}}\right\rfloor}=i_{max}$ there are $\sigma_{3}^{k_{0}}(g-i)$ possible classes of isomorphic spin graphs.  The set $\{\mathcal{S}_{i,p_{1},p_{2},p_{3}}\}$,  $i>0$, of all such  classes  has $\sigma^{2}_{4}(g+1)$ elements.( This number is equal to the number of representations of $g+1$ as a sum $k_{0}+k_{1}+k_{2}+k_{3}$ of integers that satisfy $2\leq{k_{0}}\leq{k_{1}}\leq{k_{2}}\leq{k_{3}}$.) 

Summarizing, the total number $N(3)$ of isomorphic classes of possible exceptional spin graphs with $r=3$ is equal
\begin{equation}
N(3)=\sigma_{3}(g)+\sigma^{2}_{4}(g+1)=\sigma_{4}(g+1)
\end{equation}
 
 \subsubsection{$r>3$}
 Keeping the same notation as above, an exceptional graph $\mathcal{S}_{P}$  on a surface $\Sigma$ of genus $g\geq{r+1}$ belongs to an isomorphic class $\mathcal{S}_{i,p_{1},\ldots,p_{r}}$  where $i+k_{1}+\ldots{+k_{r}}=g$ , $1+i\leq{k_{1}}\leq{\ldots}\leq{k_{r}}$ and 
 \begin{equation}
 g=(r+1)i+rp_{1}+\ldots{+2p_{r-1}+p_{r}+r}
 \end{equation}
  Now $i$ belongs to the set $\{0,1,\ldots,i_{m}\}$ with $i_{m}=\left\lfloor {\frac{g-r}{r+1}}\right\rfloor$. For the graphs for which a head $P$ is not connected with its conjugate vertex $\widetilde{P}$, i.e. for the graphs with $i=0$, the cardinality of the set $\{\mathcal{S}_{0,p_{1},\ldots,p_{r}}\}$ is equal to $\sigma_{r}(g)$. When $i>0$ then for each  $1\leq{i}\leq{i_{m}=\left\lfloor {\frac{g-r}{r+1}}\right\rfloor}$   there are $\sigma_{r}^{k_{0}}(g-i)$ possible classes of isomorphic exceptional graphs. Hence the total number of classes with $i>0$ is equal to $\sigma^{2}_{r+1}(g+1)$.Thus
  \begin{equation}
  N(r)=\sigma_{r}(g)+\sigma^{2}_{r+1}(g+1)=\sigma_{r+1}(g+1)
  \end{equation}
For the maximal possible value of $r$, i.e. for $r=g-1$ this formula gives $N(g-1)=\sigma_{g}(g+1)=1$ as expected. We see that the number of all isomorphism classes of exceptional graphs that are possible on a surface $\Sigma$ of genus $g$ is
\begin{equation}
M=M(g)=\sum^{g-1}_{r=0}N(r)=1+\sum_{r=1}^{g-1}\sigma_{r+1}(g+1)=1+\sum^{g}_{m=2}\sigma_{m}(g+1)
\end{equation}

\section{Summary}
Let $\Sigma$ be a hyperelliptic Riemann surface and let $\xi_{\epsilon}$ be any even, non singular spin bundle on this surface. The hyperelliptic involution results in the existence of interrelations between some sections of $\xi_{\epsilon}$.   These relations always link merely finite number of meromorphic sections of this bundle, each with a single, simple pole. 

In almost all cases the number of such related sections is equal to $2g+2$.  In two cases this number is $g+1$ and besides of this, we have finite number of cases when the number of interrelated sections vary. In other words, the bundle $\xi_{\epsilon}$ introduses $\epsilon$-foliation of the surface $\Sigma$ whose all leaves consist of finite number of points. These points are vertices of spin graphs. Each graph  carries all informations about the sections of $\xi_{\epsilon}$ with the unique simple pole at a given vertex of this graph (i.e. at a given point of the leaf) .

 Almost all leaves of the $\epsilon$-foliation  have $2g+2$ points that are vertices of standard graphs. These graphs  are isomorphic to each other i.e. they all belong to the unique, (standard)  class of spin graphs $\mathcal{S}(g)$. There are two leaves through the Weierstrass points $\mathcal{W}$, each consisting of $g+1$ points (they are vertices of Weierstrass graphs).

The remaining leaves of the $\epsilon$-foliation are associated with exceptional graphs.   The number of non-isomorphic classes of exceptional spin graphs that are possible on a surface of genus $g$ is equal to $M(g)$ (see the formula $(4.16)$). However, on a given surface $\Sigma$ not all types of exceptional graphs  have to occur.  Thus, we may classify hyperelliptic Riemann surfaces of genus $g$ by giving an $M(g)$-tuple of non-negative integers $m_{r,s}$, $r=0,1,\ldots,g-1$; $s=1,2,\ldots,N(r)$.  Each integer $m_{r,s}$  indicates how many exceptional graphs belonging  to a given class of isomorphic graphs, are actually present on a surface. These integers must satisfy the condition
\begin{equation}
4g=B_{0}m_{0}+\ldots{+{B_{r}}(m_{r,1}+\ldots{+m_{r,N(r)}})}+\ldots{+B_{g-1}m_{g-1}}
\end{equation} 
 where $B_{r}=2(g-r)$ is the total branch number produced by any exceptional spin graph which belongs to a class $\mathcal{S}_{i,p_{1},\ldots,p_{r}}$ with a given $r=0,1,\ldots,g-1$.
 
 This $\epsilon$-foliation of a hyperbolic Riemann surface $\Sigma$ together with the spin-graph structure of each of its leaf will allow us (in []) to attach to each point $P$ of $\Sigma$ a concrete spin group $\mathcal{G}_{P}$.


\begin{thebibliography}{8}



\bibitem{KB13}Bugajska,K.,
\emph{Standard and Weierstrass spin groups on hyperelliptic Riemann surfaces},
submitted for publication



\bibitem{KMB13}Bugajska,K.,
\emph{Exceptional spin groups on hyperelliptic Riemann surfaces},
submitted for publication


\bibitem{FK01}Farkas,H.M., ~I.Kra,
\emph{Theta constants, Riemann Surfaces and the Modular Group},
AMS, GSM vol.37, 2001

\bibitem{DM84}Mumford,D.,
\emph{Tata Lectures on Theta II},
Birkhauser, Progress in Mathematics vol.43, 1984

\bibitem{DV11}Varolin,D.,
\emph{Riemann surfaces by Way of Complex Analytic Geometry}
AMS, GSM vol.125, 2011

\bibitem{RG67}Gunning,R.,C.,
\emph{Lecture on vector bundles over Riemann surfaces}
MNPUP, Princeton,1967


\bibitem{RK95}Miranda,R.,
\emph{Algebraic curves and Riemann Surfaces},
AMS, GSM vol.5, 1995


\bibitem{FK92}Farkas,H.,M., ~I.Kra,
\emph{Riemann Surfaces},
Springer-Verlag, GTM vol.71, 1992





\end{thebibliography}
\end{document}